\documentclass[reqno,11pt,twoside]{article}

\usepackage[hmargin=1.25in,vmargin=1.25in, a4paper, centering]{geometry}
\usepackage[utf8]{inputenc}
\usepackage[T1]{fontenc}

\usepackage{amssymb, amsmath, amsthm}

\usepackage{fourier}
\usepackage{ebgaramond}
\usepackage{esint}
\usepackage{tikz-cd}
\usepackage{graphicx}
\usepackage{mathtools}
\mathtoolsset{showonlyrefs}
\usepackage{xspace}
\theoremstyle{definition}
\newtheorem{defi}{Definition}[section]

\newtheorem{remark}[defi]{Remark}

\theoremstyle{plain}
\newtheorem{theorem}[defi]{Theorem}
 \newtheorem{prop}[defi]{Proposition}
\newtheorem{lemma}[defi]{Lemma}
\newtheorem{cor}[defi]{Corollary}

\usepackage{marginnote}

\usepackage{color}

\numberwithin{equation}{section}

\usepackage[pdftitle={Steklov Spectral Flexibility},
  pdfauthor={Mikhail Karpukhin and Jean Lagacé}, 
pdfborder={0 0 0}
]{hyperref}

\makeatletter

\let\TagsLeftOn\tagsleft@true
\let\TagsLeftOff\tagsleft@false
\makeatother

\newcommand{\R}{\mathbb R}
\newcommand{\N}{\mathbb N}

\newcommand{\de}{\, \mathrm{d}}
\newcommand{\del}{\partial}

\newcommand{\Vol}{\operatorname{Vol}}

\newcommand{\CG}{\mathcal{G}}
\newcommand{\CH}{\mathcal{H}}

\newcommand{\CM}{\mathcal{M}}

\newcommand{\CW}{\mathcal{W}}

\newcommand{\rd}{\mathrm d}

\newcommand {\RL}{\mathrm L}
\newcommand{\RB}{\mathrm B}
\newcommand{\RV}{\mathrm V}

\newcommand{\RW}{\mathrm W}
\newcommand{\RC}{\mathrm C}
\newcommand{\BS}{\mathbf S}
\newcommand{\BV}{\mathbf V}

\newcommand{\normalstek}[1]{\bar \sigma_{#1}}

\newcommand{\wks}{\xrightharpoonup{*}}

\newcommand{\abs}[1]{\left\lvert #1 \right\rvert}

\newcommand{\set}[1]{\left\{ #1 \right\}}

\newcommand{\norm}[1]{\left\| #1 \right\|}
\newcommand{\bo}\boldsymbol{}
\newcommand{\bigo}[2][]{O_{#1}\left( #2 \right)}
\newcommand{\smallo}[2][]{o_{#1}\left( #2 \right)}
\newcommand{\bone}{\bo{1}}

\newcommand{\ci}{(\textbf{M1})\xspace}
\newcommand{\cii}{(\textbf{M2})\xspace}

\newcommand{\cvi}{(\textbf{M3})\xspace}
\newcommand{\csi}{(\textbf{M1})\xspace}

\newcommand{\csvi}{(\textbf{M3})\xspace}

\DeclareMathOperator{\dist}{dist}

\renewcommand{\le}{\leqslant}
\renewcommand{\ge}{\geqslant}

\newcommand{\eps}{\varepsilon}
\renewcommand{\phi}{\varphi}

\renewcommand{\S}{\mathbb S}

\renewcommand{\bar}[1]{\overline{#1}}

\renewcommand\footnotemark{}
\usepackage[nobottomtitles,pagestyles]{titlesec}
\renewpagestyle{plain}[\bfseries]{\footrule\setfoot{}{\thepage}{}}
\newpagestyle{mystyle}[\bfseries]{
\headrule
\sethead[\thepage][][Flexibility of the Steklov spectrum]{Mikhail Karpukhin and Jean Lagac\'e}{}{\thepage}
}

\renewenvironment{thebibliography}[1]{
  \begin{oldthebibliography}{#1}
      \setlength{\itemsep}{0.2em}
      \setlength{\parskip}{0.2em}
}
{
  \end{oldthebibliography}
}
\begin{document}

\title{Flexibility of Steklov eigenvalues via boundary homogenisation
}

\author{Mikhail Karpukhin \thanks{\textbf{M.K.:} Mathematics 253-37, Caltech, Pasadena, CA
        91125, USA; \newline
    \href{mailto:mikhailk@caltech.edu}{\nolinkurl{mikhailk@caltech.edu}}}
        \and
        Jean Lagac\'e \thanks{\textbf{J.L.:} Department of Mathematics, King's College London,
        The Strand, London, WC2R 2LS, UK; \newline
\href{mailto:jean.lagace@kcl.ac.uk}{\nolinkurl{jean.lagace@kcl.ac.uk}}} }
\date{}

\maketitle
\begin{abstract}
Recently, D. Bucur and M. Nahon used boundary homogenisation to show the remarkable flexibility of Steklov 
eigenvalues of planar domains. In the present paper we extend their result to higher dimensions and to 
arbitrary manifolds with boundary, even though in those cases the boundary does not generally exhibit any periodic 
structure. Our arguments use framework of variational eigenvalues and provides a different proof of the original results. 
Furthermore, we present an application of this  flexibility to the optimisation of Steklov eigenvalues under perimeter constraint.
It is proved that the best upper bound for normalised Steklov
eigenvalues of surfaces of genus zero and any fixed number of boundary
components can always be saturated by planar domains. This is the case even
though any actual maximiser (except for simply connected surfaces) is always far from
being planar themselves. In particular, it yields sharp upper bound for the first Steklov eigenvalue 
of doubly connected planar domains.

%  This paper is concerned with optimisations of the normalised Steklov
%  eigenvalues amongst planar domains with a fixed number of boundary components.
%  In a manifestation of the flexibility of the Steklov spectrum,
%  we prove that the best upper bound for normalised Steklov
%eigenvalues of surfaces of genus zero and any fixed number of boundary
%components can always be saturated by planar domains. This is the case even
%though any actual maximiser (except for simply connected surfaces) are always far from
%being planar themselves. In order to do so, we extend the boundary
%homogenisation result of Bucur--Nahon to arbitrary manifolds with boundary, even
%when the boundary does not exhibit any periodic structure. 
\end{abstract}
\thispagestyle{plain}
\pagestyle{mystyle}
\section{Introduction, main results and setting}

\subsection{Optimisation of Steklov eigenvalues}
Let $(\CM,g)$ be a complete smooth Riemannian manifold, $\Omega \subset \CM$
be a domain with non-empty Lipschitz boundary and $0 \not \equiv \beta :
\del \Omega \to [0,\infty)$ be a non-negative function. We refer to such an
$\Omega$ as a \emph{manifold with Lipschitz boundary}; any abstract manifold
with smooth boundary can be realised in this way. Consider the eigenvalue problem
\begin{equation}
  \label{prob:steklov}
\begin{cases}
  \Delta_g u = 0 & \text{in } \Omega, \\
  \del_\nu u = \sigma\beta u & \text{on } \del \Omega.
\end{cases}
\end{equation}
Under some integrability conditions on $\beta$ to be made explicit later (see
Theorem \ref{thm:boundaryhomogenisation}), the eigenvalues are discrete and form a sequence
\begin{equation}
  0 = \sigma_0(\Omega,g,\beta) < \sigma_1(\Omega,g,\beta) \le
  \sigma_2(\Omega,g,\beta) \le \dotso \nearrow \infty.
\end{equation}
For every $k$, the \emph{naturally normalised eigenvalue} is
\begin{equation}
  \label{eq:naturalnormalisation}
  \normalstek k(\Omega,g,\beta) = \sigma_k(\Omega,g,\beta) \frac{\int_{\del \Omega} \beta
  \de A_g}{\Vol_g(\Omega)^{1 - \frac 2 d}},
\end{equation}
see \cite{GKL,KM} for a discussion around the naturality of that normalisation. 
The case $\beta \equiv 1$ is of particular interest and is referred to as the
Steklov problem. The corresponding Steklov eigenvalues $\sigma_k(\Omega,g,1)$
are denoted simply as 
$\sigma_k(\Omega, g)$. 
For many known results
and open questions about the Steklov problem, the reader
can refer to the survey \cite{gpsurvey} and the references therein. 
In the present paper we are mainly concerned with the
optimisation problem for normalised Steklov eigenvalues. 

The first result of this type was obtained by Weinstock~\cite{wein} who proved
that the round disk maximises the first normalised Steklov eigenvalue in the
class of all bounded simply connected smooth planar domains. The optimisation problem for other
topologies of domains in $\mathbb{R}^2$ remains unsolved. At the same time, if
one does not impose any assumptions on the topology of the planar domain, then
the optimal upper bound for all normalised Steklov eigenvalues is
\begin{equation}
    \text{for all } k \in \N \qquad \normalstek k(\Omega,g) \le 8 \pi k,
\end{equation}
was found in~\cite{GKL}. 
The main goal of the present paper is to apply the
ideas of~\cite{bucurnahon} to the optimisation problem for
planar domains of fixed topology. Among other things, this allows us to
determine the optimal upper bound for the first normalized Steklov eigenvalue in
the class of planar domains with exactly $2$ boundary components.

As a starting point, let us note that an examination of Weinstock's proof yields
that the round disk continues to be the maximiser in the much larger class of
all simply connected Riemannian surfaces. The main observation of the present
paper is that the same holds for other topologies as well, namely, the
optimal upper bound for normalized Steklov eigenvalues for planar domains of
fixed topology does not increase after including arbitrary Riemannian surfaces
of the same topological type. To give a precise statement, for any $\gamma \ge
0$ and $b \ge 1$, we let $\Omega_{\gamma,b}$ be the compact connected surface
with boundary of genus $\gamma$ with $b$ boundary components, and we define 
$$
\Sigma_k(\gamma,b) = \sup_g\normalstek k(\Omega_{\gamma,b},g).
$$

\begin{theorem}
\label{thm:2d_optimisation}
For every $b \ge 1$ and $k\geqslant 0$ one has
$$
\sup_{\Omega \subset \R^2}\normalstek k(\Omega) =\Sigma_k(0,b),
$$ 
where the supremum is taken over the set of all bounded Lipschitz domains in $\R^2$ with
$b$ boundary components.
\end{theorem}

\begin{remark}
    It follows from our proof that in fact, $\Sigma_k(\gamma,b)$ is saturated by
domains in a surface of constant curvature for every $\gamma \ge 0$. 
\end{remark}

The quantities $\Sigma_k(\gamma,b)$ have received a lot of attention
following the influential work of Fraser and Schoen~\cite{fraschoen2}, who established the
connection between $\Sigma_k(\gamma,b)$ and free boundary minimal immersions of
$\Omega_{g,b}$ into a Euclidean ball. In particular, they showed that for the
annulus
$\Omega_{0,2}=\mathbb{A}$, $\Sigma_1(0,2)$ is achieved by a
metric $g_{cc}$ on the so-called {\em critical catenoid} in $\mathbb{B}^3$.
Combining this result with Theorem~\ref{thm:2d_optimisation} one obtains the
following.

\begin{cor}
\label{cor:annulus}
 Let $\Omega\subset \mathbb{R}^2$ be a smooth bounded domain with $2$ boundary components. Then one has
\begin{equation}
\label{ineq:annuli}
\normalstek 1(\Omega) < \normalstek 1(\mathbb{A}, g_{cc})\approx 4\pi/1.2.
\end{equation}
The inequality is sharp, i.e. there exists a sequence of domains $\Omega_n$ such
that $\normalstek 1(\Omega_n)\to  \normalstek 1(\mathbb{A}, g_{cc})$.
\end{cor}

\begin{remark}
  Theorem \ref{thm:2d_optimisation} and Corollary \ref{cor:annulus} can be
  extended to Lipschitz rather than smooth domains. In such a case, however,
  inequality~\eqref{ineq:annuli} would stop being strict. In order to rule out
  the equality case one would need to show a regularity theorem for $\normalstek
  1$-maximisers in the spirit of~\cite[Theorem 1.4]{KS}.
\end{remark}

Many sequences of planar domains saturate bound~\eqref{ineq:annuli}. 
For any bounded $\Omega\subset \mathbb{R}^2$ 
conformal to $(\mathbb{A}, g_{cc})$ one can find a maximizing sequence $\Omega_n$ such that 
$\Omega_n\to \Omega$ in Hausdorff distance. Here is a concrete example of one of those maximising
sequences, which follows from the proof of Theorem~\ref{thm:2d_optimisation} and
the geometry of $g_{cc}$. Let $t_1$ be the unique solution of $\cosh t = t$. Set
$\Omega_0 =\{z\in \mathbb{R}^2,\,\,r<|z|<R\}$, where $\log \frac{R}{r} = 2t_1$.
Then define $\Omega_n \subset \R^2$ to be the (topological) annulus whose outer boundary is the
same as $\Omega_0$, but whose inner boundary oscillates uniformly with period
$2\pi/n$, where the amplitude of the oscillations is chosen so  that the length
of the inner boundary component coincides with the length of the outer boundary.
As $n \to \infty$, the amplitude in this construction is of order
$\bigo{n^{-1}}$ and the  domains $\Omega_n$ converge in the Hausdorff metric to
$\Omega_0$ while $\normalstek 1(\Omega_n) \to \normalstek 1(\mathbb A,g_{cc})$
as $n \to \infty$.

\subsection{Flexibility of the Steklov spectrum}
Theorem~\ref{thm:2d_optimisation} can be proved by using as a main tool the material already
contained in~\cite{bucurnahon}. Despite that, we take this opportunity to give
an alternative proof using the framework of measure eigenvalues developed in
\cite{GKL}. This allows and to extend the results of~\cite{bucurnahon} to higher
dimension and in a geometric context.

We first make the observation that every compact connected smooth manifold with boundary
can be realised as a bounded smooth domain in a complete Riemannian manifold
$(\CM,g)$. Through this equivalence, we define manifolds with Lipschitz boundary
as bounded Lipschitz domains in a complete Riemannian manifold. The weighted
Steklov problem \eqref{prob:steklov} can be defined for those manifolds as well, the
normal derivative being only well-defined almost everywhere.

% Theorem~\ref{thm:2d_optimisation} is a manifestation of a more general
% phenomenon present in any dimension. 
We prove the following flexibility result for Steklov
eigenvalues, which was first observed in \cite{bucurnahon} for planar domains.

\begin{theorem}
  \label{thm:boundaryhomogenisation}
  Let $\Omega$ be a compact connected Riemannian manifold with Lipschitz
  boundary
  and let $0 \not \equiv \beta : \del \Omega \to [0,\infty)$. Suppose that
  $\beta \in \RL^{d-1}(\del \Omega)$ (if $d \ge 3$) or $\beta \in \RL \log
  \RL(\del \Omega)$ (if $d = 2$). 
  Then, there exists a family of domains $\Omega^\eps \subset \Omega$ with Lipschitz boundary such
  that
  \begin{enumerate}
    \item As $\eps \to 0$, $\del \Omega^\eps \to \del \Omega$ in the Hausdorff
      distance. 
    \item\label{it:normalisedconverge} For every $k \in \N$ the normalised
        eigenvalues $\normalstek k(\Omega^\eps,g)
      \to \normalstek k(\Omega,g,\beta)$ as $\eps\to 0$.
  \item For every $\eps > 0$, $\Omega$ and $\Omega^\eps$ have the same
      topological type.
  \end{enumerate}
\end{theorem}

As with \cite{bucurnahon}, the proof is based on homogenisation of the boundary.
However, when $d \ge 3$ the boundary may no longer carry a periodic
structure which means that classical homogenisation constructions do not work in
that setting. Instead, we adapt the geometric homogenisation ideas from \cite{GL}, which do not 
require any periodic structure. Furthermore, we interpret the statement
$\eqref{it:normalisedconverge}$ of 
Theorem~\ref{thm:boundaryhomogenisation} in the formalism of variational eigenvalues,
which in turn allows us to apply the general convergence results presented in \cite{GKL}.
In particular, this approach results in a more streamlined proof compared to~\cite{bucurnahon}.
Let us note
that boundary homogenisation of the Steklov problem in dimension $d \ge 3$ was
studied by Ferrero--Lamberti in \cite{FLstability}, however as with most
boundary homogenisation results it required domains
in Euclidean space to be of product type; we make no such geometric assumptions.

Theorem~\ref{thm:2d_optimisation} is a consequence of
Theorem~\ref{thm:boundaryhomogenisation}, Koebe uniformization theorem and
conformal invariance of Steklov eigenvalues for $d=2$. 
For $d\ge 3$, Steklov
eigenvalues are no longer conformally invariant, but one still has the following
corollary of Theorem~\ref{thm:boundaryhomogenisation}.

\begin{cor}
Let $(M,g)$ be a closed Riemannian manifold. Then for any $k\ge 0$ one has
$$
\sup_{\Omega}\normalstek k(\Omega,g) = \sup_{\Omega,\,\beta\in \RC_+(\del \Omega)}\normalstek k(\Omega,g,\beta),
$$
where $\Omega$ varies over all smooth domains $\Omega\subset M$.
\end{cor}

Informally, this corollary states that the introduction of density does not
change the optimal upper bound for the normalized Steklov eigenvalues. At the
same time, the problem with density is more natural from the geometric
viewpoint~\cite{KM}.

\subsection{Plan of the paper}

In Section \ref{sec:conformal}, we prove Theorem \ref{thm:2d_optimisation} and
its Corollary \ref{cor:annulus} using conformal changes of variable and assuming
Theorem \ref{thm:boundaryhomogenisation}. Then, in
Section \ref{sec:flex} we prove Theorem \ref{thm:boundaryhomogenisation}. This
is done by first assuming that $\Omega$ and $\beta$ are smooth, using a
geometric homogenisation procedure on the boundary. Then, we relax the
smoothness assumption and in turn approximate eigenvalues for singular
densities, then domains with Lipschitz boundary, in the end extracting a
diagonal subsequence from these procedures.

\subsection{Notation}

We make extensive use throughout the paper of Landau's asymptotic notation. We
write
\begin{itemize}
  \item indiscriminately, $f_1 = \bigo{f_2}$ or $f_1 \ll f_2$ to mean that there
  exists $C > 0$ such that $\abs{f_1} \le C f_2$;
\item $f_1 \asymp f_2$ to mean that $f_1 \ll f_2$ and $f_2 \ll f_1$;
\item $f_1 = \smallo{f_2}$ to mean that $f_1/f_2 \to 0$.
\end{itemize}
The limit in that last bullet point will be either as a parameter tends to $0$
or $\infty$ and will be clear from context. The use of a subscript, for instance
$f_1 \ll_\Omega f_2$ means that the constant $C$ or the quantities involved in
the definition of the limit may depend on the subscript.

We make use of a generalisation of $\RL^p$ spaces, called Orlicz spaces. Given
$\Phi$
an increasing, nonegative convex function on $[0,\infty)$, $\Phi(L)(\Omega)$ is the space
\begin{equation}
  \Phi(L)(\Omega) := \set{f : \Omega \to \R \text{ measurable}: \exists \eta > 0 \text{
  s.t. } \int_\Omega \Phi(\abs{f/\eta}) \de v_g < \infty}.
\end{equation}
In addition to $\Phi(x) = x^p$ (which corresponds to $\RL^p$ spaces), we also
will refer to the case $\Phi(x) = e^x$, denoted $\exp \RL$, $\Phi(x) = x
\log(1+x)$, denoted $\RL \log \RL$ which is dual to $\exp \RL$, and $\Phi(x) = x^2 \log(1+x)^{-1/2}$ denoted
$\RL^2 (\log \RL)^{-1/2}$. For a reference on Orlicz space, see
\cite{bennettsharpley}.

\subsection*{Acknowledgements} 
The authors are grateful to A. Girouard, C. Gordon, A. Hassannezhad and I. Polterovich for useful
discussions.
The research of M.K. is partially supported by the
NSF Grant DMS-2104254. The research of J.L. was partially supported by EPSRC grant
EP/T030577/1, and he is thankful for the hospitality of the University of
Bristol.

\section{Conformal changes of the metric}
\label{sec:conformal}

In this section we prove Theorem \ref{thm:2d_optimisation} and its corollary
assuming Theorem \ref{thm:boundaryhomogenisation}. We start by introducing the
notion of variational eigenvalues and look at how they behave under a conformal
change of variables. 

\subsection{Function spaces and variational eigenvalues}

We study the weighted Steklov problem \ref{prob:steklov} through the formalism
developed in \cite{GKL}, see also \cite{kok,KLP}. For any domain with Lipschitz
boundary $\Omega \subset \CM$
and any Radon measure $\mu$ supported on $\bar \Omega$, we define the Sobolev
spaces
$\RW^{1,p}(\Omega,\mu)$ as the closure of $\RC^\infty(\bar \Omega)$ under
the norm
\begin{equation}
  \norm{f}_{\RW^{1,p}(\Omega,\mu)}^p = \int_\Omega \abs{\nabla f}^p \de
  v_g + \int_\Omega \abs f^p \de \mu;
\end{equation}
we write $\RW^{1,p}(\Upsilon) := \RW^{1,p}(\Upsilon, \rd v_g)$ for the usual
Sobolev space. 

We say that a measure $\mu$ is admissible if the trace operator $T_\mu :
\RW^{1,2}(\Omega) \to \RL^2(\Omega,\mu)$ is compact. We note that under such
conditions $\RW^{1,2}(\Omega,\mu)$ is isomorphic to $\RW^{1,2}(\Omega)$, see
\cite[Theorems 3.4 and 3.5]{GKL} For an admissible
measure $\mu$ and $f \in \RC^\infty(\bar \Omega)$ we define the Rayleigh quotients
\begin{equation}
  R_{g,\mu}(f) := \frac{\int_\Omega \abs{\nabla f}^2 \de v_g}{\int_\Omega
  \abs f^2 \de \mu}.
\end{equation}
From this Rayleigh quotient we define the variational eigenvalues
\begin{equation}
\lambda_k(\Omega,g,\mu) = \inf_{F_{k+1}} \sup_{f \in F_{k+1} \setminus \set 0}
    R_{g,\mu}(f)
\end{equation}
where the infimum is taken over all $(k+1)$-dimensional subspaces $F_{k+1}
\subset \RC^\infty(\bar\Omega)$ that remain $(k+1)$-dimensional in
$\RL^2(\Omega,\mu)$. Admissibility of $\mu$ ensures that the variational
eigenvalues are discrete and form a sequence (see \cite[Proposition 4.1]{GKL})
\begin{equation}
  0 = \lambda_0(\Omega,g,\mu) < \lambda_1(\Omega,g,\mu) \le
  \lambda_2(\Omega,g,\mu) \le \dotso \nearrow \infty.
\end{equation}
The main example of variational eigenvalues employed in the present paper is the following. 
Let $0 \not \equiv \beta \in \RL^{d-1}(\del\Omega;[0,\infty))$ (if $d \ge 3$) or $\beta
\in \RL \log \RL(\del \Omega;[0,\infty))$ (if $d = 2$) and
$\CH^{d-1}\lfloor_{\del \Omega}$ be the
restriction of the Hausdorff measure to $\del \Omega$. Then
$\lambda_k(\Omega,g,\beta\CH^{d-1}\lfloor_{\del \Omega}) =
\sigma_k(\Omega,g,\beta)$ as defined in~\eqref{prob:steklov}. 

%\begin{lemma}
%  \label{lem:change}
%  Let $(\Omega,g)$ be a compact smooth Riemannian manifold, possibly
%with boundary. Let $\mu$ be a Radon measure on $\Omega$ and $F\colon
%  \Omega \to \Omega$ be a diffeomorphism, such that both $\mu$ and $F_* \mu$ are
%  admissible on $(\Omega,g)$ and $(\Omega,F_* g)$ respectively.
%  Then for any $k\geqslant 0$ one has $\lambda_k(\Omega,g,\mu) =
%  \lambda_k(\Omega,F_* g,F_*\mu)$. 
%  If $\Omega$ is a surface and $F$ is
%  conformal, then $\lambda_k(\Omega,g,\mu)= \lambda_k(\Omega,g,F_*\mu)$
%\end{lemma}
%\begin{proof}
%  The general case follows directly from the definition of the variational eigenvalues and
%  the fact that diffeomorphisms preserve smooth functions. For $\Omega$ a
%  surface and $F$ conformal, we use in addition the fact that the Dirichlet
%  energy is a conformal invariant in dimension $2$.
%\end{proof}

\subsection{Conformal optimisation}

We are now ready to prove the optimisation theorems for $d = 2$ under the
assumption of Theorem \ref{thm:boundaryhomogenisation}.
\begin{proof}[Proof of Theorem \ref{thm:2d_optimisation}]
    Let $(\Omega_{0,b},g)$ be a surface with Lipschitz boundary of genus $0$
    with $b$ boundary components. To prove our claim, it is sufficient to find a
    family of domains $\Omega^\eps \subset \R^2$ with $b$ boundary components so that
    $\normalstek k(\Omega^\eps,g_0) \to \normalstek k(\Omega_{0,b},g)$.

    By Koebe's uniformisation theorem \cite{koebe},
    there exists a circle domain $\Omega \subset \R^2$ (i.e. a domain whose
    boundary is disjoint union of circles) and a conformal diffeomorphism $\phi
    : \Omega \to \Omega_{0,b}$ such that $g_0 = \phi^* g$. It follows from
    \cite[Theorem 1.6]{KLP} that $\normalstek k(\Omega_{0,b},g) =
    \normalstek k(\Omega,g_0,\abs{\rd\phi})$ for every $k \in \N$. Furthermore, it
    follows from the proof of \cite[Lemma 5.1]{baratchartbourgeoisleblond} that
    there is $p > 1$ so that $\abs{\rd
    \phi} \in \RL^p(\del \Omega)$ . Therefore, by
    Theorem \ref{thm:boundaryhomogenisation} there exists a sequence of domains
    $\Omega^\eps \subset \Omega$ with the same topological type so that
    $\normalstek k(\Omega^\eps,g_0) \to \normalstek k(\Omega,g_0,\abs{\rd
    \phi})$ as $\eps \to 0$, concluding the proof.
\end{proof}

\begin{proof}[Proof of Corollary~\ref{cor:annulus}] The inequality~\eqref{ineq:annuli} and 
its sharpness follows immediately from Theorem~\ref{thm:2d_optimisation}
and~\cite[Theorem 1.3]{fraschoen2}. 
It remains to show that the equality can not be achieved by a smooth domain $\Omega$. 
Suppose that it does, then by~\cite[Theorem 1.3]{fraschoen2} 
there exists $\omega\in C^\infty(\mathbb{A})$ such that $\omega = 0$ on $\partial\mathbb{A}$ 
and $(\Omega,g_0)$ is isometric to
$(\mathbb{A},e^{-2\omega}g_{cc})$, where $g_{cc}$ is a metric 
on a free boundary minimal annulus in $\mathbb{B}^3$. Then the formula for Gauss
curvature in a conformal metric implies that $\omega$ is solution to the
following problem
\begin{equation}
\label{eq:Kgcc}
\begin{cases}
\Delta_{g_{cc}}\omega = -K_{g_{cc}}&\text{ on $\mathbb{A}$};\\
\omega = 0&\text{ on $\partial\mathbb{A}$.}
\end{cases}
\end{equation} 
Let $\kappa$ and $\kappa_{cc}$ be the geodesic curvture of $\Omega$ and critical
catenoid respectively. Recall that the isometry group of the critical catenoid
acts transitively on its boundary. Thus, $\kappa_{cc}$ is constant. Similarly,
since the solution to~\eqref{eq:Kgcc} is unique, the function $\partial_\nu\omega$
is also constant on $\partial \mathbb{A}$. Then one has $\kappa =
\kappa_{cc}-\partial_\nu\omega$ is also constant. The only curves of constant
geodesic curvature $\kappa$ on $\mathbb{R}^2$ are circles of radius
$\kappa^{-1}$. Hence $\partial\Omega$ consists of two circles of the same
radius, which is impossible.
\end{proof}

\section{Flexibility of the spectrum}
\label{sec:flex}

In this section we prove Theorem~\ref{thm:boundaryhomogenisation}, first under the
assumptions that $\partial\Omega$ is smooth and $\beta>0$ is a smooth density,
then under the weaker assumption that $\del \Omega$ is Lipschitz and $\beta$ is
in an appropriate integrability class.  
The section is organised as follows. First, we describe the boundary
homogenisation construction yielding the appropriate domains $\Omega^\eps$.
Then, we briefly recall abstract tools defined in \cite{GKL} 
to study eigenvalue continuity results, and we use them in order to obtain
continuity of the Steklov eigenvalues of $\Omega^\eps$ to weighted Steklov
eigenvalue on $\Omega$. Finally, we extend the results to the rough case.

\subsection{Boundary homogenisation}

\label{sec:boundaryhomogenisation}
This construction combines elements found in \cite[Section
2]{GL} (for the geometric distribution of the perturbations) and in
\cite{bucurnahon} (for the type of perturbation). A distinction from the
construction in \cite{bucurnahon} is that the approximation is done ``from the
inside'', allowing us to perform the construction intrinsically in the geometric
setting. In this subsection we assume that $\Omega$ has smooth boundary and
$0<\beta\in \RC^\infty(\partial\Omega)$. This assumption will be relaxed later in
Section~\ref{sec:Lipschitz}.
Invariance of normalised eigenvalues under scaling of the density allows us
to furthermore assume that  $\beta \ge 1$.

Let $h$ be the induced metric on $\del \Omega$, and assume that $\eps > 0$ is small enough that $h$ is
uniformly almost Euclidean in balls of radius $3 \eps$. In other words assume that in
geodesic polar coordinates around any $z \in \del \Omega$, $h$ reads
\begin{equation}
  h(\rho,\theta) = \de \rho^2 + \rho^2 g_{\S^{d-2}} + r(\rho,\theta)
\end{equation}
where $g_{\S^{d-2}}$ is the round metric on the $d-2$-dimensional sphere, and
$r$ is a symmetric $2$-tensor such that
\begin{equation}
  \norm{r}_{\RC^1(B_\eps(z))} = \bigo[\Omega]{\eps}.
\end{equation}

For every $\eps > 0$, let $\BS^\eps$ be a maximal $\eps$-separated subset of
$\del \Omega$ and let $\BV^\eps$ be the Vorono\u{\i} tesselation associated with
$\BS^\eps$, i.e. $\BV^\eps := \set{V_z^\eps : z \in \BS^\eps}$, where
\begin{equation}
  V_z^\eps := \set{x \in \del \Omega : \dist(x,z) \le \dist(x,y) \text{ for all
  } y \in \BS^\eps}
\end{equation}
and the distance is computed with respect to the metric $h$. We construct a sequence of domains $\Omega^\eps \subset \Omega$ in the following way.
For every $z \in \BS^\eps$ and $\theta \in \S^{d-2}$, let $\rho_{\theta,z}$ be
the distance from $z$ to $\del V_z^\eps$ along the geodesic starting with
direction
$\theta$. Then, define $w_z^\eps : V_z^\eps \to \R$ as
\begin{equation}
  w_z^\eps(\rho,\theta) = \eps\left( 1 - \frac{\rho}{\rho_{\theta,z}} \right).
\end{equation}
Then, $w_z^\eps$ is piecewise smooth, vanishes on $\del V_z^\eps$ and satisfies
the estimates
\begin{equation}
  \norm{w_z^\eps}_{\infty} = \eps \qquad \text{and} \qquad \norm{\nabla
  w_z^\eps}_{\infty} \asymp 1.
\end{equation}
For any smooth nonnegative function $\alpha : \del \Omega \to \R$, we have that 
\begin{equation}
  \nabla (\alpha w_z^\eps) = \alpha \nabla w_z^\eps + \bigo{\eps}.
\end{equation}
In a neighbourhood of size $2 \eps \norm{\alpha}_\infty$ of the boundary $\del
\Omega$, write Fermi coordinates as $x = (y,t)$, where $t$ is the distance along
the unit speed geodesic normal to the boundary at $y$. Define
\begin{equation}
  Q_z^\eps := \set{(y,t) : y \in V_z^\eps \text{ and } t < \alpha(y)
  w_z^\eps(y)}
\end{equation}
and
\begin{equation}
  Z_z^\eps := \set{(y,t) : y \in V_z^\eps \text{ and } t = \alpha(y)
  w_z^\eps(y)}.
\end{equation}
Finally, we define $\Omega^\eps$ as
\begin{equation}
  \Omega^\eps := \Omega \setminus \bigcup_{z \in \BS^\eps} Q_z^\eps,
\end{equation}
which has boundary
\begin{equation}
  \del \Omega^\eps = \bigcup_{z \in \BS^\eps} Z_z^\eps.
\end{equation}
We note that the family $\Omega^\eps$ has equi-Lipschitz boundary,
with the constant depending only on $g, \del \Omega,$ and $\alpha$. Furthermore,
\begin{equation}
  \Vol_g(\Omega \setminus \Omega^\eps) \ll \eps.
\end{equation}
Finally, for almost every $y\in V_z^\eps$, if $x = (y,t) \in Z_z^\eps$ then the area
element of $\del \Omega^\eps$ at $x$ is given by
\begin{equation}
  \label{eq:areaelement}
  \de A_{\del \Omega^\eps} \big|_{x} = \left(\sqrt{1 + \alpha^2 \abs{\nabla
  w_z^\eps}^2} + \bigo{\eps}\right) \de A_{\del \Omega} \big|_{y}. 
\end{equation}
We choose 
\begin{equation}
  \label{eq:alphachoose}
  \alpha = \left( \frac{\beta^2 - 1}{\abs{\nabla w_z^\eps}} \right)^{1/2},
\end{equation}
which we can do since we assumed $\beta \ge 1$.

\subsection{Continuity of eigenvalues --- the smooth setting}
We start by introducing conditions under which which variational eigenvalues are
continuous with respect to the measures used to define them. For $n \in \N$, let $\Omega_n \subset \Omega$. Let $\mu_n, \mu$ be Radon measures supported respectively
on $\Omega_n, \Omega$, we introduce the following three conditions:

\begin{itemize}
  \item[\ci] $\mu_n\wks\mu$ as measures on $\Omega$ and $\Vol_g(\Omega \setminus
  \Omega_n) \to 0$;
  \item[\cii] the measures $\mu$, $\mu_n$ are admissible for all
    $n$;
  \item[\cvi] there is an equibounded family of extension maps $J_n : \CW^{1,2}(\Omega_n,\mu_n)
    \to \CW^{1,2}(\Omega,\mu_n)$.
\end{itemize}

The following proposition appears as \cite[Proposition 4.11]{GKL}.

\begin{prop}
  \label{prop:weakenough}
  Suppose that $\Omega_n \subset \Omega$ is a sequence of domains and $\mu,
  \mu_n$ are Radon measures on respectively $\Omega, \Omega_n$
  satisfying \csi--\csvi. If $d \ge 3$, assume that $\mu_n \to \mu$ in
  $\RW^{1,\frac{d}{d-1}}(\Omega)^*$. If $d=2$, assume that $\mu_n \to \mu$ in 
  $\RW^{1,2,-1/2}(\Omega)^*$. Then, for all $k \in \N$
  \begin{equation}
    \lim_{n \to \infty} \lambda_k(\Omega_n,\mu_n) = \lambda_k(\Omega,\mu).
  \end{equation}
\end{prop}
\begin{remark}
 The space $\RW^{1,2,-1/2}(\Omega)$ is the space of all functions in $\RL^2
 (\log \RL)^{-1/2}$ such that their distributional gradient also belongs in that
 space. It is a space which is contained $\RW^{1,p}(\Omega)$ for all $1 \le p <
 2$ so that the convergence in the previous theorem can be verified in the dual
 of any of those spaces.
\end{remark}

\begin{proof}[Proof of Theorem \ref{thm:boundaryhomogenisation} under smoothness
  assumptions]

Our goal is
to apply Proposition \ref{prop:weakenough} with 
\begin{equation}
\mu_\eps =
\CH^{d-1}\bigg|_{\del \Omega^\eps} \text{ and } \mu = \beta\CH^{d-1}\bigg|_{\del
\Omega}.
\end{equation}
It is a simple
observation to see that $\Vol({\Omega^\eps}) \to \Vol(\Omega)$, and
\eqref{eq:areaelement} and \eqref{eq:alphachoose} tell us that $\mu^\eps(\Omega)
\wks \mu(\Omega)$, so that Condition \ci is verified.

Condition \cii follows from the trace inequality and the fact that $\beta\geqslant 1$. Condition \cvi follows from the fact that
for all $\eps$, $\Omega^\eps$ are Lipschitz domains whose Lipschitz constant is
controlled by $C_\Omega \sup_{x \in \Omega} \beta(x)$, and $C_\Omega$ depends on
$\Omega$ through the metric. 

It only remains to show that $\mu_\eps \to \mu$ in $\RW^{1,p}(\Omega)^*$ for all
$p>1$. Let $N^\eps$ be a $2\eps \norm{\alpha}_\infty$-tubular neighbourhood of
$\del \Omega$, 
so that $\del \Omega^\eps \subset N^\eps$. It is sufficient to
show that for every $f \in \RW^{1,1}(\Omega)$,
\begin{equation}
  \label{eq:w11enough}
  \langle \mu_\eps - \mu, f \rangle_{\RW^{1,1}(\Omega)} \le c
  \norm{f}_{\RW^{1,1}(N_\eps)}
\end{equation}
for some $c > 0$. Indeed, it follows from~\eqref{eq:w11enough} that for $f \in
\RW^{1,p}(\Omega)$, $p>1$,
\begin{equation}
  \langle \mu_\eps - \mu, f \rangle_{\RW^{1,p}(\Omega)} \le c
  \norm{f}_{\RW^{1,1}(N_\eps)} \le c \eps^{\frac{p-1}{p}}
    \norm{f}_{\RW^{1,p}(\Omega)}.
\end{equation}
By density, it is sufficient to prove \eqref{eq:w11enough} assuming that $f$ is
of class $\RC^1$. Write
\begin{equation}
  \label{eq:firststep}
  \begin{aligned}
  \langle \mu_\eps - \mu,f \rangle_{\RW^{1,1}(\Omega)} &= \int_{\del \Omega^\eps}
  f \de A_{\del \Omega^\eps} - \int_{\del \Omega} f \beta \de A_{\del \Omega} \\
  &=  \sum_{z \in \BS^\eps} \left[\int_{Z_z^\eps} f \de A_{\del \Omega^\eps} - \int_{V_z^\eps} f
  \beta \de A_{\Omega}\right].
\end{aligned}
\end{equation}
For any $t \in [0,2\eps \norm{\alpha}_\infty)$ and $y \in \del \Omega$ write
\begin{equation}
  \label{eq:ftc}
  f(y,t) = f(y,0) + \int_0^t \del_s f(y,s) \de s. 
\end{equation}
It follows from \eqref{eq:areaelement} and \eqref{eq:alphachoose} that for all
$z \in \BS^\eps$, 
\begin{equation}
  \begin{aligned}
    \int_{Z_z^\eps} f \de A_{\del \Omega^\eps} &=  \int_{V_z^\eps} f(y,\alpha(y)
    w(y) ) (\beta(y) + \bigo{\eps}) \de A_{\del \Omega} \\
    &=  \int_{V_z^\eps}  (\beta + \bigo{\eps})f \de A_{\del \Omega} +
    \int_{V_z^\eps} (\beta + \bigo{\eps}) \int_0^{t(y)} \del_s f(y,s) \de s \de
    A_{\del \Omega},
\end{aligned}
\end{equation}
where $t(y) = \alpha(y)w(y)$.
This means that we can rewrite \eqref{eq:firststep} as
\begin{equation}
  \langle \mu_\eps - \mu,f \rangle_{\RW^{1,1}(\Omega)} = \sum_{z \in \BS^\eps} \bigo{\eps}
  \int_{V_z^\eps} f \beta \de A_{\del \Omega} + \bigo{1} \int_{V_z^\eps}
  \int_0^{t(y)} \del_s f(y,s) \de s \de A_{\del \Omega}.
\end{equation}
We claim that the operator $T^\eps : \RW^{1,1}(N^\eps)
\to \RL^1(\del \Omega)$ has norm $\norm{T^\eps} \ll \eps^{-1}$. 
Indeed, integrating~\eqref{eq:ftc} over $t$ yields
$$
2\eps \norm{\alpha}_\infty f(y,0) = \int _0^{2\eps \norm{\alpha}_\infty} f(t)\de t - \int_0^{2\eps \norm{\alpha}_\infty}\int_0^t\del_s f(y,s) \de s\de t.
$$
Changing the order of integration and integrating over $y$ completes the proof of the claim.
%Indeed, the
%Anzellotti--Giaquinta characterisation of the norm of trace
%\cite{anzellottigiaquinta} tells us that for any $\delta > 0$,
%\begin{equation}
%  \norm{T^\eps} \ll_{\Omega} \sup_{x \in \del N^\eps} \sup \set{
%    \frac{\CH^{d-1}(\del^* E \cap \del
%  N^\eps)}{\CH^{d-1}(\del^* E \cap N^\eps)}
%  : E \subset N^\eps \cap B_\delta(x),
%\Per(E,N^\eps) < \infty}.
%\end{equation}
%Now, for small enough $\delta$ the relative isoperimetric inequality in $\Omega$
%tells us that for every $E \subset N^\eps \cap B_\delta(x)$, 
%\begin{equation}
%    \frac{\CH^{d-1}(\del^* E \cap \del
%    N^\eps)}{\CH^{d-1}(\del^* E \cap N^\eps)} \ll_\Omega \eps^{-1}.
%\end{equation}

Thus, one has 
\begin{equation}
  \abs{\sum_{z \in \BS^\eps} \bigo{\eps} \int_{V_z^\eps} f \beta \de A_{\del
  \Omega}} \ll \norm{\beta}_{\infty} \norm{f}_{\RW^{1,1}(N^\eps)}.
\end{equation}
By monotonicity, we have that 
\begin{equation}
  \sum_{z \in \BS^\eps} \bigo{1} \abs{\int_{V_z^\eps} \int_0^t \del_s f(y,s) \de s \de A_{\del
  \Omega}} \ll \int_{N^\eps} \abs{\nabla f} \de v_g 
  \le \norm{f}_{\RW^{1,1}(N^\eps)}.
\end{equation}
This completes the proof that \eqref{eq:w11enough} holds, which was enough for our
purposes, and the proof of Theorem \ref{thm:boundaryhomogenisation} under
smoothness assumptions is complete.
\end{proof}
%\subsection{Proof of Theorem~\ref{thm:2d_optimisation}}

%An additional feature of variational eigenvalues for surfaces is the fact that the numerator of the Rayleigh quotient~\eqref{eq:Rayleigh} depends only on the conformal class of the metric $g$. In particular, it implies the following conformal invariance.

%\begin{lemma}
%\label{lemma:conf_invariance}
%Let $(\Omega_1,g_1)$, $(\Omega_2,g_2)$ be two compact smooth Riemannian surfaces, possibly with (smooth) boundary. If $F\colon \Omega_1\to\Omega_2$ is a conformal map, which is a diffeomorphism up to the boundary, then for any Radon measure $\mu$ on $\Omega_1$ and any $k\geqslant 0$ one has $\lambda_k(\Omega_1,g_1,\mu) = \lambda_k(\Omega_2,g_2,F_*\mu)$.
%\end{lemma} 
%\begin{proof}
%Under the assumptions of the Lemma, the pullback $F^*\colon C^\infty(\Omega_1)\to C^\infty(\Omega_2)$ is an isomorphism. Thus, one has 
%$$
%\lambda_k(\Omega_1,g_1,\mu) = \lambda_k(\Omega_2, (F^{-1})^*g_1,F_*\mu) = \lambda_k(\Omega_2,g_2,F_*\mu),
%$$
%where the latter equality is due to the conformal invariance of the Dirichlet integral and the fact that $g_2$
%is conformal to $(F^{-1})^*g_1$.  
%\end{proof}

\subsection{Continuity of eigenvalues --- the singular setting}
\label{sec:Lipschitz}

\subsubsection{Singular densities}

We first give a condition on $\beta$ so that $\beta \CH^{d-1}\lfloor_{\del
\Omega}$ is an admissible
measure.
\begin{lemma}
  \label{lem:admissible}
  Suppose that $d \ge 3$ (respectively $d = 2$) and that $0 \not \equiv \beta
  \in \RL^{d-1}(\del \Omega;[0,\infty))$ (respectively in $\RL \log \RL(\del
  \Omega;[0,\infty))$) is a
  nonnegative function. Then, the trace $T_\beta : \RW^{1,2}(\Omega) \to
  \RL^2(\del \Omega,\beta \CG^{d-1}\lfloor_{\del \Omega})$ is compact; in other
  words $\mu_\beta = \beta \CH^{d-1}\lfloor_{\del \Omega}$ is an admissible
  measure.
\end{lemma}

\begin{proof}
    The case $d = 2$ is proven in \cite[Proposition 2.2]{KLP}, by factoring
    $T_\beta$ through the bounded trace $\RW^{1,2}(\Omega) \to \exp \RL^2(\del \Omega)$ and
    appropriate multiplication operators, so that $T_\beta$ is seen
    to be a norm limit of compact operators. The case $d \ge 3$ is dealt with in
    the same way, using instead the bounded trace $\RW^{1,2}(\Omega) \to
    \RL^{\frac{2(d-1)}{d-2}}(\del \Omega)$ given by Gagliardo's trace theorem
    \cite{gagliardo}. 
\end{proof}

\begin{prop}
  \label{prop:lpconv}
Let $d \ge 3$ (respectively $d = 2$) and let $\beta_n$ be a sequence of
non-negative densities converging in $\RL^{d-1}(\del \Omega)$ (respectively
$\RL \log \RL(\del \Omega))$ to a
  non-negative density $\beta$. Then, as $n \to \infty$ we have $\lambda_k(M,g,\beta_n\de
  A_g) \to \lambda_k(M,g,\beta \de A_g)$.
\end{prop}

\begin{proof}
  Conditions \ci--\cvi are respected, the only non-trivial one being \cii
  which follows from Lemma \ref{lem:admissible}. Let $u \in
  \RW^{1,\frac{d}{d-1}}(\Omega)$, for $d \ge 3$. Then, the embedding
  $\RW^{1,\frac{d}{d-1}}(\Omega) \to \RL^{\frac{d-1}{d-2}}(\del\Omega)$ given by
  Gagliardo's trace theorem \cite{gagliardo} and H\"older's inequality with
  exponents $d-1$ and $\frac{d-1}{d-2}$ yield
  \begin{equation}
    \abs{\int_{\del \Omega} u (\beta_n - \beta) \de A_g} \ll_{d,\Omega}
      \norm{u}_{\RW^{1,\frac{d}{d-1}}(\Omega)} \norm{\beta_n -
      \beta}_{\RL^{d-1}(\del \Omega)}.
  \end{equation}
  This precisely means that $\beta_n \de A_g \to \beta \de A_g$ in
  $\RW^{1,\frac{d}{d-1}}(\Omega)^*$, so that the eigenvalues converge. For $d =
  2$, the same proof holds replacing Gagliardo's trace theorem with the trace
  operator $\RW^{1,2,-1/2}(\Omega) \to \exp \RL(\del \Omega)$, see \cite[Theorem
  5.3]{cianchipick}, and H\"older's inequality on
  $\exp \RL(\del \Omega)$ and $\RL \log \RL(\del \Omega)$. 
\end{proof}

\subsubsection{Lipschitz boundary}

In order to study convergence of eigenvalues of domains with Lipschitz boundary,
we need \cite[Theorem 4.1]{BGT}. Note that this result is proven in
the Euclidean setting, but its proof extends to the Riemannian setting
directly, see \cite[Lemma 3.1]{GL} for an adaptation to
the Riemannian setting of the only part of the proof which is not completely
local.
\begin{prop}
  \label{prop:bgt}
  Let $\Omega$ be a manifold with Lipschitz boundary and for all $n \in \N$, let
  $\Omega_n \subset \Omega$ be Lipschitz domains such that $\bone_{\Omega_n} \to
  \bone_\Omega$, strongly in $\RL^1(\Omega)$, $\CH^{d-1}(\del \Omega_n) \to \CH^{d-1}(\del
  \Omega)$ and
  \begin{equation}
  \sup_{n} \norm{T_n}_{\RB\RV(\Omega_n) \to \RL^1(\del \Omega_n)} < \infty
  \end{equation}
  where $T_n$ is the trace operator. Then, for every $k \in \N$,
  $\sigma_k(\Omega_n,g) \to \sigma_k(\Omega,g)$.
\end{prop}

\begin{proof}[Proof of Theorem \ref{thm:boundaryhomogenisation} for manifolds
    with Lipschitz boundary]

We first prove that we can exhaust any compact manifold with Lipschitz boundary with a
sequence of domains with smooth boundary in such a way that the Steklov
eigenvalues are stable. 

Let $\Omega$ be a compact manifold with Lipschitz boundary. Following
\cite[Theorem A.1]{VerchotaThesis}, (see \cite[Appendix A]{mitreataylor} for a
discussion of the adaptation to the Riemannian case),  there exists a sequence of smooth
domains $\Omega_n \subset \Omega$ converging to $\Omega$ such that the
boundaries
of $\Omega$, $\Omega_n$ may be respectively parametrised by a finite number of
equi-Lipschitz maps
$\gamma_j$, $\gamma_{j,n}$ such that $\gamma_{j,n} \to \gamma_j$ uniformly. This
implies in particular that if $T^n : \RB\RV(\Omega_n) \to \RL^1(\del \Omega_n)$
is the trace operator, their norms remains uniformly bounded since it can be
estimated in terms of the Lipschitz constants of $\del \Omega_n$ and the volume
of $\Omega$ \cite{anzellottigiaquinta}. In particular, it follows from
Proposition \ref{prop:bgt} that for all $k \in \N$, $\sigma_k(\Omega_n,g) \to
\sigma_k(\Omega,g)$. 

It also follows from \cite[Theorem A.1]{VerchotaThesis} that there are
bi-Lipschitz homeomorphisms $\Phi_n$ from $\del \Omega_n$ to $\del \Omega$, whose
bi-Lipschitz character is preserved uniformly in $n$. In particular, $\Phi_n$ 
induces an isomorphism $\Phi_n^* : \RL^p(\del \Omega) \to \RL^p(\del \Omega)$
for every $p \in [1,\infty]$, whose norms are uniformly bounded in $n$.
Therefore, extracting a diagonal subsequence from
\begin{itemize}
\item first finding a
sequence of domains $\Omega_n$ with smooth boundary converging to $\Omega$;
\item then approximating
the weight $\beta \circ \Phi_n \in \RL^{d-1}(\del \Omega_n)$ by smooth weights
$\beta_{m,n}$;
\item finally finding a sequence of domains with Lipschitz boundary $\Omega^{\eps}_{m,n}$ so that
    $\CH^{d-1}\lfloor_{\del \Omega^\eps_{m,n}}$ converges to $\beta_{m,n}
    \CH^{d-1} \lfloor_{\del \Omega_n}$;
\end{itemize}
provides us with the required sequence of domains proving our claim.
\end{proof}
\bibliographystyle{alpha}
\bibliography{flex}

\end{document}